\documentclass{amsart}
\usepackage{amssymb}
\usepackage{amsthm}
\usepackage{amsmath}
\usepackage{amsfonts}
\usepackage{amsxtra}
\usepackage{color}
\usepackage[foot]{amsaddr}

\numberwithin{equation}{section}

\theoremstyle{plain}
\newtheorem{theorem}{Theorem}[section]
\newtheorem{proposition}[theorem]{Proposition}
\newtheorem{lemma}[theorem]{Lemma}

\theoremstyle{remark}
\newtheorem{remark}{{\sc Remark}}[section]

\theoremstyle{definition}

\newcommand{\reft}[1]{Theorem~\ref{#1}}
\newcommand{\refp}[1]{Proposition~\ref{#1}}
\newcommand{\refl}[1]{Lemma~\ref{#1}}

\newcommand{\refr}[1]{Remark~\ref{#1}}

\newcommand{\refeq}[1]{Eq.~\eqref{#1}}

\def\alp{\alpha}
\def\bet{\beta}
\def\eps{\varepsilon}
\def\sig{\sigma}
\def\tet{\vartheta}
\def\vphi{\varphi}
\def\vsig{\varsigma}
\def\ome{\omega}
\def\Gam{\Gamma}
\def\Lam{\varLambda}
\def\VPhi{\varPhi}
\def\VPsi{\varPsi}
\def\Ome{\Omega}
\def\btet{\boldsymbol{\vartheta}}
\def\bOme{\boldsymbol{\Ome}}
\def\bome{\boldsymbol{\ome}}
\def\bxi{\boldsymbol{\xi}}
\def\CA{\mathcal{A}}
\def\CB{\mathcal{B}}
\def\CC{\mathcal{C}}
\def\CE{\mathcal{E}}
\def\CK{\mathcal{K}}
\def\CL{\mathcal{L}}
\def\CR{\mathcal{R}}
\def\CV{\mathcal{V}}
\def\CW{\mathcal{W}}
\def\C{\mathbb{C}}
\def\F{\mathbb{F}}
\def\N{\mathbb{N}}
\def\Z{\mathbb{Z}}
\def\fq{\mathbb{F}_{q}}
\def\x{\times}
\def\UL{\operatorname{UL}}
\def\k{\Bbbk}
\def\ox{\otimes}
\newcommand{\map}[3]{#1 \colon #2 \to #3}
\newcommand{\set}[2]{\{ #1 \colon #2 \}}
\def\irr{\operatorname{Irr}}
\def\inv{^{-1}}
\def\+{\oplus}
\def\nm{\operatorname{Nm}}
\def\Cl{\operatorname{Cl}}
\def\cf{\operatorname{cf}}
\def\Sh{\operatorname{Sh}}
\def\SCl{\operatorname{SCl}}
\def\scf{\operatorname{scf}}
\def\SCh{\operatorname{SCh}}
\def\SMod{\operatorname{SMod}}
\def\AC{\widehat{\CA}}
\def\sset{\subseteq}
\newcommand{\inc}[2]{#1 \hookrightarrow #2}
\def\Hom{\operatorname{Hom}}
\def\End{\operatorname{End}}
\def\cx{\C^{\x}}
\def\ind{\operatorname{Ind}}
\newcommand{\frob}[2]{\langle #1 , #2 \rangle}
\def\tr{\operatorname{Tr}}
\newcommand{\iso}[3]{#1 \colon #2 \overset{\sim}{\to} #3}
\newcommand{\seq}[2]{#1_{1}, \ldots, #1_{#2}}
\def\FInd{\operatorname{F-Ind}}
\def\ovl{\overline}
\def\all{\text{ for all }}

\allowdisplaybreaks

\makeatletter
\@namedef{subjclassname@2020}{\textup{2020} Mathematics Subject Classification}
\makeatother

\begin{document}


\title[Shintani descent for standard supercharacters]{Shintani descent for standard supercharacters of algebra groups}

\author[C. Andr\'e]{Carlos A. M. Andr\'e$^{1,\dagger}$}
\author[A.L. Branco Correia]{Ana L. Branco Correia$^{1,2}$}
\author[J. Dias]{Jo\~ao Dias$^{1}$}
 \address{$^1$Centro de An\'alise Funcional, Estruturas Lineares e Aplica\c c\~oes (Grupo de Estruturas Lineares e Combinat\'orias) \\ Departamento de Matem\'atica \\ Faculdade de Ci\^encias da Universidade de Lisboa \\ Campo Grande, Edif\'\i cio C6, Piso 2 \\ 1749-016 Lisboa \\ Portugal}
 \address{$^2$Centro de Investiga\c c\~ao, Desenvolvimento e Inova\c c\~ao da Aca\-de\-mia Mi\-li\-tar (CINAMIL) \& Departamento de Ci\^encias Exatas e Engenharias \\ Academia Militar \\ Av. Conde Castro Guimar\~aes \\ 2720-113 Amadora \\ Portugal}
 \address{$^{\dagger}$Corresponding author}
 \email[C. Andr\'e]{caandre@ciencias.ulisboa.pt}
 \email[A.L. Correia]{alcorreia@ciencias.ulisboa.pt} 
 \email[J. Dias]{joaodias104@gmail.com}
%
%

\thanks{This research was made within the activities of the Group for Linear, Algebraic and Combinatorial Structures of the Centre for Functional Analysis, Linear Structures and Applications (University of Lisbon, Portugal), and was partially supported by the Portuguese Science Foundation (FCT) through the Strategic Projects UID/MAT/04721/2013 and UIDB/04721/2020.}

\subjclass[2020]{20C15, 20D15, 20G40}

\date{April 10, 2023}

\keywords{Shintani descent; Supercharacter; Superclass function; Algebra group; Superdual algebra}

\begin{abstract}
Let $\CA(q)$ be a finite-dimensional nilpotent algebra over a finite field $\fq$ with $q$ elements, and let $G(q) = 1+\CA(q)$. On the other hand, let $\k$ denote the algebraic closure of $\fq$, and let $\CA = \CA(q) \otimes_{\fq} \k$. Then $G = 1+\CA$ is an algebraic group over $\k$ equipped with an $\fq$-rational structure given by the usual Frobenius map $\map{F}{G}{G}$, and $G(q)$ can be regarded as the fixed point subgroup $G^{F}$. For every $n \in \N$, the $n$th power $\map{F^{n}}{G}{G}$ is also a Frobenius map, and $G^{F^{n}}$ identifies with $G(q^{n}) = 1 + \CA(q^{n})$. The Frobenius map restricts to a group automorphism $\map{F}{G(q^{n})}{G(q^{n})}$, and hence it acts on the set of irreducible characters of $G(q^{n})$. Shintani descent provides a method to compare $F$-invariant irreducible characters of $G(q^{n})$ and irreducible characters of $G(q)$. In this paper, we show that it also provides a uniform way of studying supercharacters of $G(q^{n})$ for $n \in \N$. These groups form an inductive system with respect to the inclusion maps $G(q^{m}) \to G(q^{n})$ whenever $m \mid n$, and this fact allows us to study all supercharacter theories simultaneously, to establish connections between them, and to relate them to the algebraic group $G$. Indeed, we show that Shintani descent permits the definition of a certain ``superdual algebra'' which encodes information about the supercharacters of $G(q^{n})$ for $n \in \N$.
\end{abstract}

\maketitle


\section{Introduction} \label{intro}

Throughout this paper, we fix a prime number $p$, as well as a power $q$ of $p$. We denote by $\fq$ the finite field with $q$ elements, and fix a finite-dimensional associative nilpotent $\fq$-algebra $\CA(q)$; in particular, $\CA(q)$ does not have an identity. Moreover, we let $G(q) = 1+\CA(q)$ be the set of formal objects of the form $1+a$ where $a \in \CA(q)$. Then $G(q)$ is easily seen to be a group with respect to the multiplication defined by $(1+a)(1+b) = 1+a+b+ab$ for all $a,b \in \CA(q)$. Following \cite{Isaacs1995a}, a group $G(q)$ constructed in this way will be referred to as the \textit{algebra group} associated with $\CA(q)$. As a well-known example, if $\CA(q) = \mathcal{UL}_{n}(q)$ is the $\fq$-algebra consisting of all strictly upper triangular $n \x n$ matrices over $\fq$, then the corresponding algebra group $G(q) = 1+\CA(q)$ is isomorphic to the unipotent linear group $\UL_{n}(q)$ consisting of all unipotent upper-triangular $n \x n$ matrices over $\fq$.

Let $\k$ denote the algebraic closure of $\fq$, let $\CA = \CA(q) \ox_{\fq} \k$ be the nilpotent $\k$-algebra obtained from $\CA(q)$ by extension of scalars, and let $G = 1+\CA$. Then $G$ is an algebraic group over $\k$ equipped with an $\fq$-rational structure given by the usual Frobenius map $\map{F}{G}{G}$. Moreover, the algebra group $G(q)$ can be regarded as the fixed point subgroup $G^{F} = \set{g \in G}{F(g) = g}$. Let $n \in \N$ be any positive integer, and consider the $n$th power $\map{F^{n}}{G}{G}$ of $F$. Then $F^{n}$ is also a Frobenius map for an $\F_{q^{n}}$-rational structure on $G$, and the fixed point subgroup $G^{F^{n}}$ is naturally identified with the algebra group $G(q^{n})$ associated with the nilpotent $\F_{q^{n}}$-algebra $\CA(q^{n}) = \CA(q) \ox_{\fq} \F_{q^{n}}$. Furthermore, the Frobenius map $\map{F}{G}{G}$ restricts to a group automorphism (of order $n$) $\map{F}{G(q^{n})}{G(q^{n})}$, and hence it acts on the set of irreducible characters of $G(q^{n})$. Shintani descent provides a method to compare $F$-invariant irreducible characters of $G(q^{n})$ and irreducible characters of $G(q)$. Henceforth, we denote by $\irr(G(q))$ the set of irreducible characters of $G(q)$, and by $\irr(G(q^{n}))^{F}$ the set of $F$-invariant irreducible characters of $G(q^{n})$.
 
More precisely, we consider the action of $G(q^{n})$ on itself given by $F$-twisted conjugation: $g \cdot h = ghF(g\inv)$ for all $g,h \in G(q^{n})$. Using Lang's theorem, we can define the \textit{norm map} $\map{\nm_{n}}{\Cl_{F}(G(q^{n}))}{\Cl(G(q))}$ which is a bijection between the set $\Cl_{F}(G(q^{n}))$ of $F$-conjugacy classes of $G(q^{n})$ and the set $\Cl(G(q))$ of conjugacy classes of $G(q)$. For details on the definition of $\nm_{n}$, see the paragraph below \refl{extension}. On the other hand, let $\cf_{F}(G(q^{n}))$ denote the complex vector space of functions $G(q^{n}) \to \C$ which take a constant value on each $F$-conjugacy class, and let $\cf(G(q))$ denotes the complex vector space of functions $G(q) \to \C$ which take a constant value on each conjugacy class. The mapping $\vphi \mapsto \vphi \circ \nm_{n}\inv$ defines a $\C$-linear isomorphism $\map{\Sh_{n}}{\cf_{F}(G(q^{n}))}{\cf(G(q))}$ which we call the \textit{Shintani descent} from $G(q^{n})$ to $G(q)$. For every $F$-invariant irreducible character of $\chi \in \irr(G(q^{n}))^{F}$, we can define a function $\widetilde{\chi} \in \cf_{F}(G(q^{n}))$ (which is well-defined up to scaling by $n$th roots of unity) such that the set $\set{\widetilde{\chi}}{\chi \in \irr(G(q^{n}))^{F}}$ forms an orthonormal basis of $\cf_{F}(G(q^{n}))$ (with respect to the usual Hermitian inner product). The Shintani image $\set{\Sh_{n}(\widetilde{\chi})}{\chi \in \irr(G(q^{n}))^{F}}$ is an orthonormal basis of $\cf(G(q))$ which, in general, is not equal to the orthonormal basis $\irr(G(q))$.

The irreducible characters of finite algebra groups are hard to describe. Thus, it might be of interest to consider smaller and more manageable families of characters which are still rich enough to provide some relevant information about their representations. This question was first addressed in the context of the finite unitriangular groups $\UL_{n}(q)$ by Andr\'e in \cite{Andre1995a} where supercharacters were named ``basic characters'' (see also \cite{Andre2001a,Andre2002a}), and later in the context of arbitrary finite groups by Diaconis and Isaacs \cite{Diaconis2008a}, where the notion of a supercharacter theory was formalised. By a \textit{supercharacter theory} of a finite group $G$ we mean a pair $(\CK,\CE)$ where $\CK$ is a partition of $G$, and $\CE$ is an orthogonal set of characters of $G$ satisfying:
\begin{enumerate}
\item $|\CK| = |\CE|$,
\item every character $\xi \in \CE$ takes a constant value on each member $K \in \CK$, and
\item each irreducible character of $G$ is a constituent of one of the characters $\xi \in \CE$.
\end{enumerate}
We refer to the members $K \in \CK$ as \textit{superclasses} and to the characters $\xi \in \CE$ as \textit{supercharacters} of $G$. Note that superclasses of $G$ are always unions of conjugacy classes. Moreover, $\{1\}$ is a superclass and the trivial character $1_{G}$ is always a supercharacter of $G$. The superclasses and the supercharacters of a particular supercharacter theory exhibit much of the same duality as the conjugacy classes and the irreducible characters of the group do, and thus a supercharacter theory can be interpreted as an ``approximation'' of the classical character theory. Indeed, the usual character theory of $G$ is a trivial example of a supercharacter theory where $\CK$ is the set $\Cl(G)$ of conjugacy classes of $G$ and $\CE$ is the set $\irr(G)$ of irreducible characters of $G$.

The \textit{standard supercharacter theory} of an arbitrary finite algebra group is described in \cite{Diaconis2008a}; see also \cite{Andre1999a,Andre2008a,Yan2001a}. The superclasses of $G(q)$ are easy to describe; since there is no danger of ambiguity, we will abbreviate the terminology and refer to a standard superclass simply as a superclass of $G$. The direct product $\Gam(q) = G(q) \x G(q)$ acts naturally on the left of $\CA(q)$ by the rule $(g,h) \cdot a = gah\inv$ for all $g,h \in G(q)$ and all $a\in\CA(q)$. Then $\CA(q)$ is partitioned into $\Gam(q)$-orbits $\Gam(q) \cdot a$ for $a \in \CA(q)$, and this determines a partition of $G(q)$ into subsets $1 + \Gam(q) \cdot a$ for $a \in \CA(q)$. These subsets are precisely what we define as the \textit{superclasses} of $G(q)$. We use the notation $\SCl(G(q))$ to the denote the set of superclasses of $G(q)$.

We next define the set of supercharacters of $G(q)$; as in the case of superclasses, by a supercharacter of $G(q)$ we will always understand a standard supercharacter. A function $\map{\vphi}{G(q)}{\C}$ is a \textit{superclass function} if it is constant on the superclasses. We denote by $\scf(G(q))$ the complex vector space of superclass functions defined on $G(q)$; it is clear that $\scf(G(q))$ is a vector subspace of $\cf(G(q))$. The supercharacters of $G(q)$ form a distinguished orthogonal basis of $\scf(G(q))$ which we will denote by $\SCh(G(q))$. Similarly to the case of superclasses, the elements of $\SCh(G(q))$ are parametrised by the orbits of the contragradient action of $\Gam(q) = G(q) \x G(q)$ on the Pontryagin dual $\AC(q)$ of the additive group of $\CA(q)$. For every $\tet \in \AC(q)$ and every $g,h\in G(q)$, we define $(g,h) \cdot \tet \in \AC(q)$ by the rule $((g,h) \cdot \tet)(a) = \tet(g\inv ah)$ for all $a \in \CA(q)$. For every $\Gam(q)$-orbit $\ome \sset \AC$, we define a function $\map{\xi_{\ome}}{G(q)}{\C}$ whose normalisation 
$\map{\widetilde{\xi}_{\ome}}{G(q)}{\C}$ is given by
\begin{equation} \label{FSCh}
\widetilde{\xi}_{\ome}(1+a) = \frac{1}{|\ome|} \sum_{\tet \in \ome} \tet(a)
\end{equation}
for all $a \in \CA(q)$. (For every character $\map{\chi}{G(q)}{\C}$, we define the \textit{normalised character} $\map{\widetilde{\chi}}{G(q)}{\C}$ to be the function given by $\widetilde{\chi} = \chi(1)\inv \chi$.) We set $\SCh(G(q)) = \set{\xi_{\ome}}{\ome \in \Ome}$ where $\Ome$ denotes the set of $\Gam(q)$-orbits on $\AC(q)$. By \cite[Theorem~5.6]{Diaconis2008a}, we know that $\SCh(G(q))$ is the set of supercharacters of a supercharacter theory of $G(q)$ whose set of superclasses is $\SCl(G(q))$.

The Frobenius automorphism $\map{F}{G(q^{n})}{G(q^{n})}$ acts on the set $\SCh(G(q^{n}))$ of supercharacters of $G(q^{n})$ (see \refp{invariant}). As in the case of irreducible characters, Shintani descent allows us to compare the set $\SCh(G(q^{n}))^{F}$ of $F$-invariant supercharacters of $G(q^{n})$ with $\SCh(G(q))$. Indeed, every $F$-invariant supercharacter of $G(q^{n})$ takes a constant value on each $F$-conjugacy class of $G(q^{n})$ (see \refl{extension}), and so $\SCh(G(q^{n}))^{F}$ is an orthogonal basis of $\scf_{F}(G(q^{n}))$. As a consequence, the Shintani image $\set{\Sh_{n}(\xi)}{\xi \in \SCh(G(q^{n}))^{F}}$ is an orthogonal basis of $\scf(G(q))$. We will prove that the Shintani image of an $F$-invariant supercharacter of $G(q^{n})$ is in fact a supercharacter of $G(q)$, and thus we conclude that the mapping $\xi \mapsto \Sh_{n}(\xi)$ defines a bijection between $\SCh(G(q^{n}))^{F}$ and $\SCh(G(q))$ (see \reft{main}).

Shintani descent also provides a uniform way of studying supercharacters of the distinct algebra groups $G(q^{n})$ for $n \in \N$. We observe that the finite groups $G(q^{n})$, for $n \in \N$, form an inductive system with respect to the natural inclusion maps $\inc{G(q^{m})}{G(q^{n})}$ whenever $m,n \in \N$ are such that $m \mid n$. This fact allows us to study all supercharacter theories simultaneously, to establish connections between them, and to relate them to the affine algebraic group $G$. We note that the inductive limit $\varinjlim_{n} G(q^{n})$ identifies with the filtered union $\bigcup_{n \in \N} G(q^{n}) = G$, and Shintani descent permits the definition of a certain \textit{``superdual algebra''} which ``contains'' $\scf(G(q^{n}))$ for all $n \in \N$, and hence encodes information about the supercharacters of all algebra groups $G(q^{n})$ for $n \in \N$.

\section{Shintani descent for supercharacters}

Let $n \in \N$ be arbitrary (but fixed), let $G(q^{n}) = 1+\CA(q^{n})$ be as in the introduction, and let $\AC(q^{n}) = \Hom(\CA^{+}(q^{n}),\cx)$ be the dual group of the additive group $\CA^{+}(q^{n})$ of $\CA(q^{n})$. For every $\tet \in \AC(q^{n})$ and every $g\in G(q^{n})$, we define $g\tet, \tet g \in \AC(q^{n})$ by the formulas $(g\vartheta)(a) = \vartheta(g\inv a)$ and $(\vartheta g)(a) = \vartheta(ag\inv)$ for all $a \in \CA(q^{n})$. Thus, for every $\tet \in \AC(q^{n})$, we have a left $G(q^{n})$-orbit $G(q^{n})\tet$, a right $G(q^{n})$-orbit $\tet G(q^{n})$, and also a two-sided $G(q^{n})$-orbit $G(q^{n})\vartheta G(q^{n})$ (we note that the left and right $G(q^{n})$-actions on $\AC(q^{n})$ commute).

Let $\tet \in \AC(q^{n})$ be arbitrary, and let $$L_{G(q^{n})}(\tet) = \set{g \in G(q^{n})}{g\tet = \tet}$$ denote the left centraliser of $\tet$ in $G(q^{n})$. We note that $L_{G(q^{n})}(\tet) = 1 + \CL_{G(q^{n})}(\tet)$ where $$\CL_{G(q^{n})}(\tet) = \set{a \in \CA(q^{n})}{\tet(au) = 1 \all u \in \CA(q^{n})}.$$ It is clear that $\CL_{G(q^{n})}(\tet)$ is a subalgebra of $\CA(q^{n})$, and hence $L_{G(q^{n})}(\tet)$ is an algebra subgroup of $G(q^{n})$. (A subgroup $H$ of $G(q^{n})$ is an \textit{algebra subgroup} if there exists a subalgebra $\CB(q^{n})$ of $\CA(q^{n})$ such that $H = 1+\CB(q^{n})$.) It is straightforward to check that the mapping $g \mapsto \tet(g-1)$ defines a linear character $\map{\nu_{\tet}}{L_{G(q^{n})}(\tet)}{\cx}$ of $L_{G(q^{n})}(\tet)$. We define the \textit{(standard) supercharacter} of $G(q^{n})$ associated with $\tet$ to be the induced character
\begin{equation} \label{supch}
\xi_{\tet} = \ind^{G(q^{n})}_{L_{G(q^{n})}(\tet)}(\nu_{\tet}).
\end{equation}
Thus, for every $g \in G$, $$\xi_{\tet}(g) = \frac{1}{|L_{G(q^{n})}(\tet)|} \sum_{h \in G(q^{n})} (\nu_{\tet})^{\circ}(hgh\inv)$$ where $\map{(\nu_{\tet})^{\circ}}{G(q^{n})}{\C}$ denotes the extension-by-zero of $\nu_{\tet}$ to $G(q^{n})$.

By \cite[Theorem~5.6]{Diaconis2008a}), we know that
\begin{equation} \label{superchar}
\xi_{\tet}(g) = \frac{|G(q^{n})\tet|}{|G(q^{n})\tet G(q^{n})|} \sum_{\tet' \in G(q^{n}) \tet G(q^{n})} \tet'(g-1)
\end{equation}
for all $g \in G(q^{n})$. From this formula, it is clear that, for every $\tet,\tet' \in \AC(q^{n})$, we have $\frob{\xi_{\tet}}{\xi_{\tet'}} = 0$ unless $G(q^{n}) \tet G(q^{n}) = G(q^{n}) \tet' G(q^{n})$, in which case $\xi_{\tet} = \xi_{\tet'}$. (If $G$ is any finite group, we denote by $\frob{-}{-}_{G}$, or simply by $\frob{-}{-}$ if there is no danger of ambiguity, the usual Hermitian inner product on the complex vector space consisting of all functions $G \to \C$.) Furthermore, it is straightforward to check that the regular character $\rho_{G(q^{n})}$ of $G(q^{n})$ decomposes as the orthogonal sum $$\rho_{G(q^{n})} = \sum_{\xi \in \SCh(G(q^{n}))} m_{\xi}\, \xi$$ where $m_{\xi} = \xi(1)\slash \frob{\xi}{\xi} \in \Z_{\geq 0}$ for all $\xi \in \SCh(G(q^{n}))$; we remark that, for every $\tet \in \AC(q^{n})$, we have $\xi_{\tet}(1) = |G(q^{n})\tet|$ and $\frob{\xi_{\tet}}{\xi_{\tet}} = |G(q^{n})\tet \cap \tet G(q^{n})|$. As a consequence of this decomposition, we conclude that every irreducible character of $G(q^{n})$ is a constituent of a unique supercharacter.

\begin{remark}
For every $\tet \in \AC(q^{n})$, we may also define the \textit{right centraliser} $$R_{G(q^{n})}(\tet) = \set{g \in G(q^{n})}{\tet g = \tet}$$ of $\tet$ in $G(q^{n})$, and all results about left centralisers may be dualised for right centralisers. In particular, $R_{G(q^{n})}(\tet)$ is an algebra subgroup of $G(q^{n})$; in fact, $$\CR_{G(q^{n})}(\tet) = R_{G(q^{n})}(\tet)-1 = \set{a \in \CA(q^{n})}{\tet(ua) = 1 \all u \in \CA(q^{n})}$$ is clearly a subalgebra of $\CA(q^{n})$. Furthermore, the mapping $g \mapsto \tet(g-1)$ defines a linear character $\map{\nu'_{\tet}}{R_{G(q^{n})}(\tet)}{\cx}$ of $R_{G(q^{n})}(\tet)$, and we may consider the induced character $\xi'_{\tet} = \ind^{G(q^{n})}_{R_{G(q^{n})}(\tet)}(\nu'_{\tet})$. As for the supercharacter $\xi_{\tet}$, it can be proved that $$\xi'_{\tet}(g) = \frac{|\tet G(q^{n})|}{|G(q^{n})\tet G(q^{n})|} \sum_{\tet' \in G(q^{n}) \tet G(q^{n})} \tet'(g-1)$$ for all $g \in G(q^{n})$. Since $|\tet G(q^{n})| = |G(q^{n})\tet|$ (see \cite[Lemma~4.2]{Diaconis2008a}), we conclude that $\xi'_{\tet} = \xi_{\tet}$.
\end{remark}

\begin{remark} \label{centraliser}
We recall that, for every $\tet \in \AC(q^{n})$, the two-sided $G(q^{n})$-orbit $G(q^{n})\tet G(q^{n})$ is defined to be the $\Gam(q^{n})$-orbit $\Gam(q^{n})\cdot \tet$ where the direct product $\Gam(q^{n}) = G(q^{n})\x G(q^{n})$ acts naturally on the left of $\AC(q^{n})$ via $(g,h)\cdot \tau = g\tau h\inv$ for all $g,h \in G(q^{n})$ and all $\tau \in \AC(q^{n})$. We note that $\Gam(q^{n})$ is the algebra group associated with the nilpotent $\F_{q^{n}}$-algebra $\CA(q^{n})\x \CA(q^{n})$, and that the centraliser $C_{\Gam(q^{n})}(\tet)$ of $\tet$ in $\Gam(q^{n})$ is given by $C_{\Gam(q^{n})}(\tet) = 1+\CC_{\Gam(q^{n})}(\tet)$ where $$\CC_{\Gam(q^{n})}(\tet) = \set{(a,b) \in \CA(q^{n})\x \CA(q^{n})}{\tet(au) = \tet(ub) \all u \in \CA(q^{n})}.$$ Since $\CC_{\Gam(q^{n})}(\tet) $ is a subalgebra of $\CA(q^{n})\x \CA(q^{n})$ in $G(q^{n})$, we see that $\CC_{\Gam(q^{n})}(\tet)$ is an algebra subgroup of $\Gam(q^{n})$.
\end{remark}

Now, let $\map{F}{G(q^{n})}{G(q^{n})}$ denote the Frobenius automorphism of $G(q^{n})$, and let $\widetilde{G}(q^{n})$ denote the semidirect product $\widetilde{G}(q^{n}) = G(q^{n}) S$ where $S = \langle F \rangle$ denotes the cyclic group generated by $F$ and where the multiplication satisfies $F g = F(g)F$ for all $g \in G(q^{n})$. Let $\tet \in \AC(q^{n})$ be $F$-invariant, and consider the induced character $\ind_{G(q^{n})}^{\widetilde{G}(q^{n})}(\xi_{\tet})$. By transitivity of induction (and by the definition of $\xi_{\tet}$), $$\ind_{G(q^{n})}^{\widetilde{G}(q^{n})}(\xi_{\tet}) = \ind_{G(q^{n})}^{\widetilde{G}(q^{n})} \big( \ind_{L_{G(q^{n})}(\tet)}^{G(q^{n})}(\nu_{\tet}) \big) = \ind_{\widetilde{L}_{G(q^{n})}(\tet)}^{\widetilde{G}(q^{n})} \big( \ind_{L_{G(q^{n})}(\tet)}^{\widetilde{L}_{G(q^{n})}(\tet)}(\nu_{\tet}) \big)$$ where $\widetilde{L}_{G(q^{n})}(\tet) = L_{G(q^{n})}(\tet)S$. We note that, since $\tet$ is $F$-invariant, $L_{G(q^{n})}(\tet)$ is an $F$-invariant subgroup of $G(q^{n})$, and hence $\widetilde{L}_{G(q^{n})}(\tet) = L_{G(q^{n})}(\tet)S$ is in fact a subgroup of $\widetilde{G}(q^{n})$. Let $\widehat{S} = \Hom(S,\cx)$ denote the dual group of $S$. For every $\tau \in \widehat{S}$, let $\map{\nu_{\tet} \cdot \tau}{\widetilde{L}_{G(q^{n})}(\tet)}{\cx}$ be defined by $(\nu_{\tet} \cdot \tau)(gF) = \nu_{\tet}(g)\tau(F)$ for all $g \in L_{G(q^{n})}(\tet)$. Since $\tet$ is $F$-invariant, it is not hard to check that $\nu_{\tet} \cdot \tau$ is a linear character of $\widetilde{L}_{G(q^{n})}(\tet)$, and that $$\ind_{L_{G(q^{n})}(\tet)}^{\widetilde{L}_{G(q^{n})}(\tet)}(\nu_{\tet}) = \sum_{\tau \in \widehat{S}} \nu_{\tet} \cdot \tau.$$ It follows that $$\ind_{G(q^{n})}^{\widetilde{G}(q^{n})}(\xi_{\tet}) = \sum_{\tau \in \widehat{S}} \ind_{\widetilde{L}_{G(q^{n})}(\tet)}^{\widetilde{G}(q^{n})}(\nu_{\tet} \cdot \tau).$$ Let $\map{\xi_{\tet} \cdot \tau}{\widetilde{G}(q^{n})}{\cx}$ be defined by $(\xi_{\tet} \cdot \tau)(g\sig) = \xi_{\tet}(g)\tau(\sig)$ for all $g \in G(q^{n})$. Then $$\xi_{\tet} \cdot \tau = \big( \ind_{L_{G(q^{n})}(\tet)}^{G(q^{n})}(\nu_{\tet}) \big) \cdot \tau = \ind_{\widetilde{L}_{G(q^{n})}(\tet)}^{\widetilde{G}(q^{n})}(\nu_{\tet} \cdot \tau)$$ for all $\tau \in \widehat{S}$, and thus $$\ind_{G(q^{n})}^{\widetilde{G}(q^{n})}(\xi_{\tet}) = \sum_{\tau \in \widehat{S}} \xi_{\tet} \cdot \tau.$$ In particular, $\widetilde{\xi}_{\tet} = \xi_{\tet} \cdot 1_{S}$ is an extension to $\widetilde{G}(q^{n})$ of $\xi_{\tet} \in \SCh(G(q^{n}))$, and we deduce that $$\xi_{\tet}(g\inv hF(g)) = \widetilde{\xi}_{\tet}(g\inv hF(g)F) = \widetilde{\xi}_{\tet}(g\inv (hF)g) = \widetilde{\xi}_{\tet}(hF) = \xi_{\tet}(h)$$ for all $g,h \in G(q^{n})$. Thus, $\xi_{\tet}$ takes a constant value on each $F$-conjugacy class of $G(q^{n})$, that is, $\xi_{\tet} \in \cf_{F}(G(q^{n}))$. Indeed, we have the following result.

\begin{lemma} \label{extension}
If $\tet \in \AC(q^{n})$ is $F$-invariant, then $$\xi_{\tet}(g) = \frac{1}{|L_{G(q^{n})}(\tet)|} \sum_{x\in G(q^{n})} (\nu_{\tet})^{\circ}(x\inv gF(x))$$ for all $g \in G(q^{n})$.
\end{lemma}

\begin{proof}
Let $\tet \in \AC(q^{n})$ be $F$-invariant, and consider the extension $\widetilde{\nu}_{\tet} = \nu_{\tet} \cdot 1_{S}$ of $\nu_{\tet}$ to $\widetilde{L}_{G(q^{n})}(\tet)$. Then $\widetilde{\xi}_{\tet} = \ind_{\widetilde{L}_{G(q^{n})}(\tet)}^{\widetilde{G}(q^{n})} (\widetilde{\nu}_{\tet})$ is an extension of $\xi_{\tet}$ to $\widetilde{G}(q^{n})$. Let $g \in G(q^{n})$ be arbitrary. Then $\xi_{\tet}(g) = \widetilde{\xi}_{\tet}(gF)$, and thus $$\xi_{\tet}(g) = \frac{1}{|\widetilde{L}_{G(q^{n})}(\tet)|} \sum_{1 \leq m \leq n} \sum_{x\in G} (\widetilde{\nu}_{\tet})^{\circ}((xF^{m})\inv (gF)(xF^{m})).$$ For every $1 \leq m \leq n$ and every $x \in G$, $$(xF^{m})\inv (gF)(xF^{m}) = \big(F^{-m}x\inv gF(x)F^{m}\big)F = F^{-m}(x\inv gF(x))\cdot F,$$ and thus $(xF^{m})\inv (gF)(xF^{m}) \in \widetilde{L}_{G(q^{n})}(\tet)$ if and only if $F^{-m}(x\inv gF(x)) \in L_{G(q^{n})}(\tet)$. Since $L_{G(q^{n})}(\tet)$ is $F$-invariant and $$\widetilde{\nu}_{\tet}(F^{-m}(x\inv gF(x))\cdot F) = \nu_{\tet}(F^{-m}(x\inv gF(x))) = \nu_{\tet}(x\inv gF(x))$$ for all $1 \leq m \leq n$ and all $x\in G$ such that $x\inv gF(x) \in L_{G(q^{n})}(\tet)$, we conclude that
\begin{align*}
\xi_{\tet}(g) &= \frac{1}{|\widetilde{L}_{G(q^{n})}(\tet)|} \sum_{1 \leq m \leq n} \sum_{x\in G} (\nu_{\tet})^{\circ}(x\inv gF(x)) \\ &= \frac{1}{|L_{G(q^{n})}(\tet)|} \sum_{x\in G} (\nu_{\tet})^{\circ}(x\inv gF(x)),
\end{align*}
as required.
\end{proof}

We next recall the definition of the \textit{norm map} $\map{\nm_{n}}{\Cl_{F}(G(q^{n}))}{\Cl(G(q))}$. Let $\CK \in \Cl_{F}(G(q^{n}))$ be an $F$-conjugacy class of $G(q^{n})$, and let $g \in \CK$ be arbitrary. By Lang's theorem (see \cite[Corollary~V.16.4]{Borel1991a}), there exists $h \in G$ such that $g = h\inv F(h)$. Then $F^{n}(h)h\inv \in G(q)$, and we define $\nm_{n}(\CK) \in \Cl(G)$ to be the conjugacy class $$\nm_{n}(\CK) = \set{x\inv(F^{n}(h)h\inv) x}{x \in G(q)}$$ of $G(q)$. We observe that $F^{n}(h)h\inv = h\nm_{n}(g)h\inv$ where $$\nm_{n}(g) = g F(g) F^{2}(g) \cdots F^{n-1}(g);$$ however, $\nm_{n}(g) \in G(q^{n})$ does not necessarily lie in $G(q)$. It can be shown that $\nm_{n}(\CK)$ is independent of the choices of $g \in \CK$ and of $h \in G$ such that $g = h\inv F(h)$, and that we get a well-defined bijection $\map{\nm_{n}}{\Cl_{F}(G(q^{n}))}{\Cl(G(q))}$ given by the mapping $\CK \mapsto \nm_{n}(\CK)$ for $\CK \in \Cl_{F}(G(q^{n}))$. For a detailed proof, see for example \cite[Lemma~2.2]{Kawanaka1977a}.

The mapping $\vphi \mapsto \vphi \circ \nm_{n}\inv$ clearly defines a $\C$-linear isomorphism $$\map{\Sh_{n}}{\cf_{F}(G(q^{n}))}{\cf(G(q))}$$ which we call the \textit{Shintani descent} from $G(q^{n})$ to $G(q)$. By \cite[Lemma~1.10]{Shoji1992a}, we know that the Shintani descent is an isometry in the sense that $$\frob{\Sh_{n}(\vphi)}{\Sh_{n}(\psi)}_{G(q)} = \frob{\vphi}{\psi}_{G(q^{n})}$$ for all $\vphi,\psi \in \cf_{F}(G(q^{n}))$. In particular, if we denote by $\AC(q^{n})^{F}$ the set of all $F$-invariant characters in $\AC(q^{n})$, then \refl{extension} and the orthogonality of supercharacters imply that $\set{\Sh_{n}(\xi_{\tet})}{\tet \in \AC(q^{n})^{F}}$ is an orthogonal subset of $\cf(G(q))$. Our goal is to prove that this set equals the set $\SCh(G(q))$ of supercharacters of $G(q)$, and that $\Sh_{n}$ defines by restriction a bijection $\map{\Sh_{n}}{\SCh(G(q^{n}))^{F}}{\SCh(G(q))}$ where $\SCh(G(q^{n}))^{F}$ denotes the set of $F$-invariant supercharacters of $G(q^{n})$. In particular, we see that $\Sh_{n}(\SCh(G(q^{n}))^{F}) = \SCh(G(q))$ is an orthogonal $\C$-basis of the vector space $\scf(G(q))$ of all superclass functions of $G(q)$.

Our next result establishes that in fact $\SCh(G(q^{n}))^{F} = \set{\Sh_{n}(\xi_{\tet})}{\tet \in \AC(q^{n}))^{F}}$. For the proof, it is convenient to realise $\AC(q^{n})$ as the set $\set{\tau \circ f}{f \in \CA^{\ast}(q^{n})}$ where $\map{\tau}{(\F_{q^{n}})^{+}}{\cx}$ is an arbitrarily fixed non-trivial character of the additive group of $\F_{q^{n}}$ and $\CA^{\ast}(q^{n})$ denotes the dual $\F_{q^{n}}$-vector space of $\CA(q^{n})$. Indeed, it is well-known that the mapping $f \mapsto \tau \circ f$ defines a group isomorphism $\CA^{\ast}(q^{n}) \cong \AC(q^{n})$. This isomorphism is clearly $\Gam(q^{n})$-equivariant where we consider the natural action of $\Gam(q^{n}) = G(q^{n})\x G(q^{n})$ on the left of $\CA^{\ast}(q^{n})$ given by $((g,h)\cdot f) = f(g\inv ah)$ for all $g,h \in G(q^{n})$, all $f \in \CA^{\ast}(q^{n})$ and all $a \in \CA(q^{n})$. Therefore, for every $f \in \CA^{\ast}(q^{n})$, the two-sided $G(q^{n})$-orbit $G(q^{n}) (\tau \circ f) G(q^{n})$ of $\tau \circ f \in \AC(q^{n})$ is given by $$\Gam(q^{n}) \cdot (\tau \circ f) = \tau \circ (\Gam(q^{n}) \cdot f) = \set{\tau \circ (x \cdot f)}{x \in \Gam(q^{n})};$$ notice that $\tau \circ (x \cdot f) = x \cdot (\tau \circ f)$ for all $x \in  \Gam(q^{n})$. Finally, we observe that, if $\CA^{\ast}$ denotes the dual $\k$-vector space of $\CA$, then the Frobenius map $\map{F^{n}}{\CA}{\CA}$ induces a Frobenius map $\map{F^{n}}{\CA^{\ast}}{\CA^{\ast}}$ (given by the mapping $f \mapsto f \circ F^{n}$), so that $\CA^{\ast}(q^{n})$ identifies with the subspace $(\CA^{\ast})^{F^{n}} = \set{f \in \CA^{\ast}}{F^{n}(f) = f}$ of $\CA^{\ast}$ consisting of $F^{n}$-fixed elements. Moreover, the group $\Gam = G \x G$ acts naturally on $\CA^{\ast}$, and thus, for every $f \in \CA^{\ast}(q^{n})$, we may consider the subset $(\Gam \cdot f)^{F^{n}}$ consisting of all $F^{n}$-fixed elements of the $\Gam$-orbit $\Gam \cdot f$; we observe that this subset is clearly $\Gam(q^{n})$-invariant, and hence it is a disjoint union of $\Gam(q^{n})$-orbits.

\begin{proposition} \label{invariant}
Let $\xi \in \SCh(G(q^{n}))$ be an arbitrary supercharacter of $G(q^{n})$. Then $\xi^{F} = \xi \circ F$ is also a supercharacter of $G(q^{n})$, and $\xi$ is $F$-invariant if and only if there exists an $F$-invariant character $\tet \in \AC(q^{n})^{F}$ such that $\xi = \xi_{\tet}$.
\end{proposition}

\begin{proof}
Let $\tet \in \AC(q^{n})$ be such that $\xi = \xi_{\tet}$, and recall from \cite[Theorem~5.8]{Diaconis2008a} that $$\xi_{\tet}(1+a) = \frac{|G(q^{n})\tet|}{|G(q^{n}) a G(q^{n})|} \sum_{b \in G(q^{n}) a G(q^{n})} \tet(b)$$ for all $a \in \CA(q^{n})$. Since $F(G(q^{n}) a G(q^{n})) = G(q^{n}) F(a) G(q^{n})$, we conclude that $$\xi_{\tet}(1+F(a)) = \frac{|G(q^{n})\tet|}{|G(q^{n}) a G(q^{n})|} \sum_{b \in G(q^{n}) a G(q^{n})} \tet(F(b)) = \xi_{\tet \circ F}(a)$$ for all $a \in \CA(q^{n})$, and thus $\xi^{F} = \xi_{\tet\circ F}$ is also a supercharacter of $G(q^{n})$. Moreover, if $\tet$ is $F$-invariant, then $\tet \circ F = \tet$, and so $(\xi_{\tet})^{F} = \xi_{\tet\circ F} = \xi_{\tet}$ is also $F$-invariant.

Conversely, let $\tet \in \AC(q^{n})$ be such that $\xi_{\tet}$ is $F$-invariant, that is, $\xi_{\tet} \circ F = \xi_{\tet}$. Since $\xi_{\tet} \circ F = \xi_{\tet \circ F}$, we must have $\tet \circ F \in G(q^{n}) \tet G(q^{n})$, and thus the two-sided orbit $G(q^{n}) \tet G(q^{n})$ is $F$-invariant. Let $f \in \CA^{\ast}(q^{n})$ be such that $\tet = \tau \circ f$, and consider the $\Gam$-orbit $\Gam \cdot f$. Since $G(q^{n})\tet G(q^{n})$ is $F$-invariant, it is clear that $\Gam(q^{n}) \cdot f$ is $F$-invariant, and thus the $\Gam$-orbit $\Gam \cdot f  \sset \CA^{\ast}$ is also $F$-invariant. Since $\Gam$ is connected,  \cite[Corollary~V.16.5]{Borel1991a} implies that there is $x \in \Gam$ such that $F(x \cdot f) = x\cdot f$. We claim that $x\cdot f \in \Gam(q^{n}) \cdot f$ which will imply that the character $\tau \circ (x\cdot f) \in \AC(q^{n})$ is $F$-invariant and lies in $\Gam(q^{n}) \cdot \tet = G(q^{n}) \tet G(q^{n})$.

To prove our claim, we consider the $\Gam(q^{n})$-orbit $\Gam(q^{n}) \cdot (x\cdot f) \sset (\Gam \cdot f)^{F^{n}}$ which contains $x\cdot f$; we note that $x\cdot f \in (\CA^{\ast})^{F^{n}}$. Since the $\Gam(q^{n})$-orbits $\Gam(q^{n}) \cdot f$ and $\Gam(q^{n}) \cdot (x\cdot f)$ are both contained in $(\Gam \cdot f)^{F^{n}}$, it is enough to prove that $(\Gam \cdot f)^{F^{n}}$ is indeed a single $\Gam(q^{n})$-orbit. To see this, we start by noting that $x\cdot f = F^{n}(x\cdot f) = F^{n}(x)\cdot f$, and thus $x\inv F^{n}(x) \in C_{\Gam}(f)$ where $C_{\Gam}(f)$ denotes the centraliser of $f$ in $\Gam$. Let $y \in \Gam$ be such that $F^{n}(y\cdot f) = y \cdot f$ and suppose that there exists $z \in \Gam(q^{n})$ such that $y \cdot f = z \cdot (x \cdot f)$. Then $y \cdot f = (zx) \cdot f$, and hence $y\inv zx \in C_{\Gam}(f)$. Furthermore, $$y\inv zx(x\inv F^{n}(x))F((y\inv zx)\inv) = y\inv z F(z\inv)F(y) = y\inv F(y)$$ which means that the elements $x\inv F(x)$ and $y\inv F(y)$ of $C_{\Gam}(f)$ are $F^{n}$-conjugate. It follows that there is a bijection between $\Gam(q^{n})$-orbits on $(\Gam \cdot f)^{F^{n}}$ and $F^{n}$-conjugacy classes of $C_{\Gam}(f)$. By \cite[Lemma~2.5]{Srinivasan1979a}, we know that $F^{n}$-conjugacy classes of $C_{\Gam}(f)$ are in one-to-one correspondence with $F^{n}$-conjugacy classes of the quotient group $C_{\Gam}(f) \slash C_{\Gam}(f)^{\circ}$ where $C_{\Gam}(f)^{\circ}$ denotes the connected component of $C_{\Gam}(f)$. As in \refr{centraliser}, we see that $C_{\Gam}(f)$ is an algebra subgroup of $\Gam$, and hence it is connected. It follows that there is a unique $F^{n}$-conjugacy class in $C_{\Gam}(f)$, and thus there is a unique $\Gam(q^{n})$-orbit on $(\Gam \cdot f)^{F^{n}}$. Therefore, $\Gam(q^{n}) \cdot (x\cdot f) = \Gam(q^{n}) \cdot f$, and thus $x\cdot f \in \Gam(q^{n}) \cdot f$, as claimed.

To conclude the proof, it is enough to observe that $\tet' = \tau \circ (x\cdot f)$ lies in $G(q^{n})\tet G(q^{n})$, and hence $\xi_{\tet'} = \xi_{\tet}$.
 \end{proof}

Next, we identify the $F$-invariant characters in $\AC(q^{n})$. Let $\map{\tr_{n}}{\CA(q^{n})}{\CA(q)}$ denote the trace map defined by $$\tr_{n}(a) = a + F(a) + F^{2}(a) + \cdots + F^{n-1}(a)$$ for all $a \in \CA(q^{n})$. It follows from Lang's theorem that $\map{\tr_{n}}{\CA(q^{n})}{\CA(q)}$ is surjective, and thus the mapping $\tau\mapsto \tau \circ \tr_{n}$ defines an injective group homomorphism $\map{\widehat{\tr}_{n}}{\AC(q)}{\AC(q^{n})}$. Furthermore, since the (affine) algebraic group $\CA^{+}$ is connected, this map induces a group isomorphism $\iso{\widehat{\tr}_{n}}{\AC(q)}{\AC(q^{n})^{F}}$ where $\AC(q^{n})^{F} = \set{\tet \in \AC(q^{n})}{\tet \circ F = \tet}$. (For a proof see \cite[Section~4]{Lusztig2006a}.) In particular, we conclude that
\begin{equation} \label{invar1}
\AC(q^{n})^{F} = \set{\tau \circ \tr_{n}}{\tau \in \AC(q)}.
\end{equation}

\begin{proposition}
Let $\tet \in \AC(q^{n})$ be $F$-invariant, and let $\tau \in \AC(q)$ be such that $\tet = \tau \circ \tr_{n}$. Then $$L_{G(q)}(\tau) = L_{G(q^{n})}(\tet)^{F}\quad \text{and} \quad \Sh_{n}(\nu_{\tet}) = \nu_{\tau}$$ where $\map{\Sh_{n}}{\cf_{F}(L_{G(q^{n})}(\tet))}{\cf(L_{G(q)}(\tau))}$ is the Shintani descent.
\end{proposition}

\begin{proof}
Since $\tr_{n}(ga) = g\tr_{n}(a)$, it is clear that $\tau(\tr_{n}(ga)) = \tau(g\tr_{n}(a))$ for all $g \in G(q)$ and all $a \in \CA(q^{n})$. On one hand, if $g \in L_{G(q)}(\tau)$, then $$\tet(ga) = \tau(\tr_{n}(ga)) = \tau(g\tr_{n}(a)) = \tau(\tr_{n}(a)) = \tet(a),$$ and thus $g \in L_{G(q^{n})}(\tet)^{F}$. On the other hand, let $g \in L_{G(q^{n})}(\tet)^{F}$. Then $g \in G(q)$, and so $$\tau(\tr_{n}(a)) = \tau(\tr_{n}(ga)) = \tau(g\tr_{n}(a))$$ for all $a \in \CA(q^{n})$. Since $\map{\tr_{n}}{\CA(q^{n})}{\CA(q)}$ is surjective, we conclude that $\tau(b) = \tau(gb)$ for all $b \in \CA(q)$, and hence $g \in L_{G(q)}(\tau)$. It follows that $ L_{G(q^{n})}(\tet)^{F} = L_{G(q)}(\tau)$, as required.

Now, let $g \in L_{G(q^{n})}(\tet)$, and let $h \in L_{G}(\tet)$ be such that $g = h\inv F(h)$; here, we set $L_{G}(\tet) = 1+\CL_{G}(\tet)$ where $\CL_{G}(\tet) = \CL_{G(q^{n})}(\tet) \ox_{\F_{q^{n}}} \k$ (hence, $L_{G(q^{n})}(\tet) = L_{G}(\tet)^{F^{n}}$). Since $L_{G(q^{n})}(\tet))^{F} = L_{G(q)}(\tau)$, we have the norm map $\map{\nm_{n}}{\Cl_{F}(L_{G(q^{n})}(\tet))}{\Cl(L_{G(q)}(\tau))}$ which sends the $F$-conjugacy class $\CK \sset L_{G(q^{n})}(\tet)$ of $g$ to the conjugacy class $\nm_{n}(\CK) \sset L_{G(q)}(\tau)$ of $F^{n}(h)h\inv$. By definition of the Shintani descent $\map{\Sh_{n}}{\cf_{F}(L_{G(q^{n})}(\tet))}{\cf(L_{G(q)}(\tau))}$, we see that $\Sh_{n}\inv(\nu_{\tau}) = \nu_{\tau} \circ \nm_{n}$, and so $$\Sh_{n}\inv(\nu_{\tau})(g) = \nu_{\tau}(F^{n}(h)h\inv) = \nu_{\tau}(h\nm_{n}(g)h\inv)$$ where $\nm_{n}(g) = g F(g) \cdots F^{n-1}(g) \in L_{G(q^{n})}(\tet)$. Since $h\nm_{n}(g)h\inv = F^{n}(h)h\inv$ and $\nm_{n}(g)$ are both elements of $L_{G(q^{n})}(\tet)$, \cite[I,~3.4]{Springer1970a} implies that there exists $k \in L_{G(q^{n})}(\tet)$ such that $h\nm_{n}(g)h\inv = k\nm_{n}(g)k\inv$, and thus $$\Sh_{n}\inv(\nu_{\tau})(g) = \nu_{\tau}(k\nm_{n}(g)k\inv).$$

Let $\widetilde{\tau} \in \AC(q^{n})$ be an (arbitrary) extension of $\tau \in \AC(q)$ to $\CA(q^{n})$, and observe that $L_{G(q)}(\tau) \sset L_{G(q^{n})}(\widetilde{\tau})$. Indeed, if $\{\seq{e}{r}\}$ is an $\fq$-basis of $\CA(q)$, then $$\CL_{G(q)}(\tau) = \set{a \in \CA(q)}{\tau(ae_{i}) = 1 \all 1 \leq i \leq r}.$$ Since $\{\seq{e}{r}\}$ is an $\F_{q^{n}}$-basis of $\CA(q^{n})$, $$\CL_{\widetilde{\tau}} = \set{a \in \CA}{\widetilde{\tau}(ae_{i}) = 1 \all 1 \leq i \leq r}.$$ By the choice of $\widetilde{\tau}$, we know that $\widetilde{\tau}(ae_{i}) = \tau(ae_{i}) = 1$ for all $a \in \CL_{G(q)}(\tau)$ and all $1 \leq i \leq r$, and so $\CL_{G(q)}(\tau) \sset \CL_{G(q^{n})}(\widetilde{\tau})$. It follows that $$L_{G(q)}(\tau) = 1+\CL_{G(q)}(\tau) \sset 1+\CL_{G(q^{n})}(\widetilde{\tau}) = L_{G(q^{n})}(\widetilde{\tau}),$$ as required. Since $\nu_{\widetilde{\tau}}(1+a) = \widetilde{\tau}(a) = \tau(a) = \nu_{\tau}(1+a)$ for all $a \in L_{\tau}$, we see that the character $\map{\nu_{\widetilde{\tau}}}{L_{G(q^{n})}(\widetilde{\tau})}{\cx}$ is an extension of $\map{\nu_{\tau}}{L_{G(q)}(\tau)}{\cx}$ to $L_{G(q^{n})}(\widetilde{\tau})$, and thus $$\nu_{\tau}(k\inv\nm_{n}(g)k) = \nu_{\widetilde{\tau}}(k\inv\nm_{n}(g)k).$$

Finally, since $\CL_{G(q^{n})}(\tet)$ is the $\F_{q^{n}}$-linear span of $\CL_{G(q^{n})}(\tet)^{F} = \CL_{G(q)}(\tau)$, we must have $\CL_{G(q^{n})}(\tet) \sset \CL_{G(q^{n})}(\widetilde{\tau})$, and hence $L_{G(q^{n})}(\tet) \sset L_{G(q^{n})}(\widetilde{\tau})$. In particular, both $k$ and $g$ are elements of $L_{G(q^{n})}(\widetilde{\tau})$, and thus we deduce that
\begin{align*}
\nu_{\widetilde{\tau}}(k\inv\nm_{n}(g)k) &= \nu_{\widetilde{\tau}}(\nm_{n}g)) = \nu_{\widetilde{\tau}}(g) \nu_{\widetilde{\tau}}(F(g)) \cdots \nu_{\widetilde{\tau}}(F^{n-1}(g)) \\ &= \widetilde{\tau}(g-1) \widetilde{\tau}(F(g-1)) \cdots \widetilde{\tau}(F^{n-1}(g-1)) \\ &= \widetilde{\tau}\big((g-1) + F(g-1) + \cdots + F^{n-1}(g-1)\big) \\ &= \widetilde{\tau}(\tr_{n}(g-1)) = \tau(\tr_{n}(g-1)) \\ &= \nu_{\tau\circ\tr_{n}}(g) = \nu_{\tet}(g).
\end{align*}
Therefore, $\Sh_{n}\inv(\nu_{\tau}) = \nu_{\tet}$, and hence $\Sh_{n}(\nu_{\tet}) = \nu_{\tau}$, as required.
\end{proof}

In what follows, we determine the Shintani descent $\Sh_{n}(\xi) \in \cf(G(q))$ of an $F$-invariant supercharacter $\xi \in \SCh(G(q^{n}))^{F}$; we recall that $\xi \in \cf_{F}(G)$ (by \refl{extension}). By \refp{invariant}, we know that there exists an $F$-invariant character $\tet \in \AC(q^{n})^{F}$ such that $\xi = \xi_{\tet}$. Moreover, by \refeq{invar1}, $\tet = \tau \circ \tr_{n}$ for some $\tau \in \AC(q)$. We claim that $\Sh_{n}(\xi_{\tet}) = \xi_{\tau}$ or, equivalently, that $$\xi_{\tet} = \xi_{\tau\circ\tr} = \xi_{\tau} \circ \nm_{n}.$$ In fact, since $\xi_{\tau} = \ind_{L_{G(q)}(\tau)}^{G(q)}(\nu_{\tau})$ and since $\nu_{\tau} \circ \nm_{n} = \Sh_{n}\inv(\nu_{\tau}) = \nu_{\tet}$ (by the previous proposition), \cite[Proposition~1.14]{Shoji1992a} implies that $$\xi_{\tau} \circ \nm_{n} = \FInd_{L_{G(q^{n})}(\tet)}^{G(q^{n})}(\nu_{\tau} \circ \nm_{n}) = \FInd_{L_{G(q^{n})}(\tet)}^{G(q^{n})}(\nu_{\tet})$$ where $$\FInd_{L_{G(q^{n})}(\tet)}^{G(q^{n})}(\nu_{\tet})(g) = \frac{1}{|L_{G(q^{n})}(\tet)|} \sum_{x \in G(q^{n})} (\nu_{\tet})^{\circ}(x\inv gF(x))$$ for all $g \in G$. By \refl{extension}, we conclude that $\FInd_{L_{G(q^{n})}(\tet)}^{G(q^{n})}(\nu_{\tet}) = \xi_{\tet}$, and thus $\xi_{\tet} = \xi_{\tau} \circ \nm_{n} = \Sh_{n}\inv(\xi_{\tau})$, as claimed. This completes the proof of the following main result.

\begin{theorem} \label{main}
The Shintani descent $\map{\Sh_{n}}{\cf_{F}(G(q^{n}))}{\cf(G(q))}$ restricts to a bijection $\map{\Sh_{n}}{\SCh(G(q^{n}))^{F}}{\SCh(G(q))}$. Under this bijection, $\Sh_{n}(\xi_{\tet}) = \xi_{\tau}$ where $\tet \in \AC(q^{n})^{F}$ and $\tau \in \AC(q)$ is such that $\tet = \tau \circ \tr_{n}$.
\end{theorem}

\section{The superdual of $G$}

Shintani descent allows us to associate with the algebra group $G$ the \textit{superdual} $\scf(G)$ which encodes the (standard) supercharacter theories of all the finite algebra groups $G(q^{n})$ for $n \in \N$. For every $n \in \N$, let $\scf(G(q^{n}))$ denote the complex vector space consisting of all superclass functions of $G(q^{n})$, and recall that $\SCh(G(q^{n}))$ is a orthogonal basis of $\scf(G(q^{n}))$.

Let $m,n \in \N$ be such that $m \mid n$, let $\map{\Sh_{n,m}}{\cf_{F}(G(q^{n}))}{\cf(G(q^{m}))}$ denote the Shintani descent from $G(q^{n})$ to $G(q^{m})$, and let $\nm^{\ast}_{n,m} = \Sh_{n,m}\inv$ denote its inverse. Hence, $\map{\nm^{\ast}_{n,m}}{\cf(G(q^{m}))}{\cf_{F}(G(q^{n}))}$ is defined by the mapping $\psi \mapsto \psi \circ \nm_{n,m}$ where $\map{\nm_{n,m}}{\Cl_{F}(G(q^{n}))}{\Cl(G(q^{m}))}$ is the norm map. By \reft{main}, $\nm^{\ast}_{n,m}$ induces a bijection $\map{\nm^{\ast}_{n,m}}{\SCh(G(q^{m}))}{\SCh(G(q^{n}))^{F}}$, and hence an injective $\C$-linear map $\map{\nm^{\ast}_{n,m}}{\scf(G(q^{m}))}{\scf(G(q^{n}))}$.

It is straightforward to check that $$\nm^{\ast}_{n,m'} \circ \nm^{\ast}_{m',m} = \nm^{\ast}_{n,m}$$ whenever $m,m',n \in \N$ satisfy $m \mid m' \mid n$. Therefore, with respect to the transition maps $\map{\nm^{\ast}_{n,m}}{\scf(G(q^{m}))}{\scf(G(q^{n}))}$ whenever $m,n \in \N$ are such that $m \mid n$, the vector spaces $\scf(G(q^{n}))$, for $n \in \N$, form an inductive system, and thus we may define $\scf(G)$ to be the direct limit $$\scf(G) = \varinjlim_{n \in \N} \scf(G(q^{n})).$$ Hence, the vector space $\scf(G)$ must be regarded as the quotient of the disjoint union $\bigsqcup_{n \in \N} \scf(G(q^{n}))$ by the equivalence relation $\sim$ generated by the relation $$\vphi \sim \vphi' \quad \iff \quad \vphi' = \nm^{\ast}_{n,m}(\vphi) = \vphi \circ \nm_{n,m}$$ for all $m,n \in \N$ with $m \mid n$, all $\vphi \in \scf(G(q^{m}))$ and all $\vphi' \in\scf(G(q^{n}))$. Thus, $$\scf(G) = \set{[\vphi]}{\vphi \in \scf(G(q^{n})),\ n \in \N}$$ where $[\vphi]$ denotes the equivalence class which contains $\vphi \in \scf(G(q^{n}))$ for $n \in \N$. Finally, we note that the vector space $\scf(G)$ has a basis $$\SCh(G) = \set{[\xi]}{\xi \in \scf(G(q^{n})),\ n \in \N}$$ consisting of all equivalence classes of supercharacters of the finite algebra groups $G(q^{n})$ for $n \in \N$.

The set $\SCh(G)$ should not be confused with the set of (standard) supercharacters of $G$. In fact, since $G$ is an infinite group, the representations of $G$ are in general infinite-dimensional, and thus there are obvious obstructions to defining a supercharacter of $G$ as the trace of a representation. Andr\'e and Lochon \cite{Andre2023a} (see also \cite{Andre2018a}) used the Vershik-Kerov ergodic method (as described in \cite{Vershik1974a}; see also \cite[Section~1.5]{Kerov2003a}) to define an alternative notion of the standard normalised supercharacters in the case where $G$ is considered as a topological group with respect the discrete topology, and proved that such a standard normalised supercharacter of $G$ may be approached by standard normalised supercharacters (as defined in \refeq{FSCh}) of the finite algebra groups $G(q^{n})$ for $n \in \N$.  On the other hand, the main theorem proved by Boyarchenko in \cite{Boyarchenko2011a} suggests that known results on the standard supercharacter theory of finite algebra groups may also be valid in the context of unitary representations of locally compact algebra groups where characters should be replaced by unitary representations. In particular, for a general locally compact algebra group $G$, it might be possible to construct a \textit{``unitary super-representation''} of $G$ as a unitary representation induced from a unitary character of an appropriate algebra subgroup of $G$. The terminology ``super-representation'' is motivated by the similarity with the construction of the standard supercharacters of an arbitrary finite algebra group described in \cite[Section~5]{Diaconis2008a}. Indeed, by the results of \cite{Diaconis2008a}, the standard supercharacters of a finite algebra group are in one-to-one correspondence with the isomorphism classes of (unitary) super-representations where by a ``super-representation''  we understand a finite-dimensional representation which affords a standard supercharacter. We should mention that, as in \cite{Diaconis2008a}, our construction of the ``unitary super-representations'' of $G$ is based on a cruder version of Kirillov's orbit method (see \cite{Kirillov2004a}) where we replace coadjoint orbits by double orbits on the dual space of $\CA$.

For our purposes, it is enough to consider unitary representations of $G$ on a (Hausdorff) pre-Hilbert space, that is, a complex vector space endowed with a Hermitian inner product. For simplicity of writing, we prefer to use the equivalent notion of pre-Hilbert (left) modules over the complex group algebra $\C[G]$ of $G$. Thus, a $\C[G]$-module $\CV$ is unitary if $\CV$ is a pre-Hilbert space over $\C$ and if the corresponding representation $\map{\pi}{G}{\mathbf{U}(\CV)}$ is a continuous group homomorphism from $G$ to the group $\mathbf{U}(\CV)$ of unitary linear automorphisms of $\CV$ equipped with the strong operator topology. Here, and henceforth, we assume that the field $\k$ is equipped with a topology which is Hausdorff and such that $\k$ becomes a locally compact topological field. The topology of $\k$ induces naturally a topology in $G$ with respect to which $G$ becomes a locally compact topological group. Similarly, $\k$ induces naturally a topology in $\CA$ with respect to which $\CA$ becomes a locally compact topological ring. Moreover, the mapping $a \mapsto 1+a$ clearly defines a homeomorphism between $\CA$ and $G$. Finally, we note that, since the topology is Hausdorff, every finite fields $\F_{q^{n}}$, for $n \in \N$, are closed subsets of $\k$. Similarly, $G(q^{n})$ and $\CA(q^{n})$, for $n \in \N$, are closed subsets of $G$ and $\CA$, respectively.

Our next goal is to prove \reft{thm:super} where we establish the existence of a one-to-one correspondence between the equivalence classes in $\SCh(G)$ and the isomorphism classes of certain unitary $\C[G]$-modules which we will refer to as the \textit{super $\C[G]$-modules} (see \refeq{super} for the definition).  In order to achieve this, we first recall the definition of the Serre's dual $\AC$ of the additive group $\CA^{+}$ over $\C$. Our main reference is \cite[Section~4]{Lusztig2006a}.

Let $m,n \in \N$ be such that $m \mid n$, and let $\map{\tr_{n,m}}{\CA(q^{n})}{\CA(q^{m})}$ denote the trace map defined by $$\tr_{n,m}(a) = a + F^{m}(a) + F^{2m}(a) + \cdots + F^{n-m}(a)$$ for all $a \in \CA(q^{n})$. It follows from Lang's theorem that $\map{\tr_{n,m}}{\CA(q^{n})}{\CA(q^{m})}$ is surjective, and thus the mapping $\tet \mapsto \tet \circ \tr_{n,m}$ defines an injective group homomorphism $\map{\widehat{\tr}_{n,m}}{\AC(q^{m})}{\AC(q^{n})}$ with image $\widehat{\tr}_{n,m}(\AC(q^{m})) = \AC(q^{n})^{F^{m}}$. Since $\widehat{\tr}_{n,m'} \circ \widehat{\tr}_{m',m} = \widehat{\tr}_{n,m}$ whenever $m,m',n \in \N$ satisfy $m \mid m' \mid n$, the finite groups $\AC(q^{n})$, for $n \in \N$, form an inductive system with respect to the transition maps $\map{\widehat{\tr}_{n,m}}{\AC(q^{m})}{\AC(q^{n})}$ whenever $m,n \in \N$ are such that $m \mid n$, and we define the \textit{Serre's dual} of $\CA$ to be the direct limit $$\AC = \varinjlim_{n \in \N} \AC(q^{n}).$$ As in the case of $\scf(G)$, we view $\AC$ as the quotient of the disjoint union $\bigsqcup_{n \in \N} \AC(q^{n})$ by the equivalence relation $\sim$ generated by the relation $$\tet \sim \tet' \quad \iff \quad \tet' = \tet \circ \tr_{m,n}$$ for all $m,n \in \N$ with $m \mid n$, all $\tet \in \AC(q^{n})$ and all $\tet' \in\AC(q^{m})$. Thus, $$\AC = \set{[\tet]}{\tet \in \AC(q^{n}),\ n \in \N}$$ where $[\tet]$ denotes the equivalence class which contains $\tet \in \AC(q^{n})$ for $n \in \N$.

In what follows, we will show that $\AC$ may be embedded as a dense subspace of the Pontryagin dual $\CA^{\circ} = \Hom(\CA^{+},\C)$ consisting of all unitary characters of $\CA^{+}$. Henceforth, we fix an arbitrary $\fq$-basis $\{\seq{e}{r}\}$ of $\CA(q)$, and let $\{\seq{e^{\ast}}{r}\}$ be its dual basis of dual vector space $\CA^{\ast}(q) = \Hom_{\fq}(\CA(q),\fq)$. On the other hand, we recall that the dual group $\AC(q) = \Hom(\CA^{+},\cx)$ consists of all functions $\map{\vsig \circ f}{\CA(q)}{\cx}$ where $\map{\vsig}{\fq}{\cx}$ is an arbitrarily fixed non-trivial character of $\fq^{+}$ and $f \in \CA^{\ast}(q)$. Moreover, it is not hard to check that $\AC(q)$ has the structure of a vector space over $\fq$ with respect to the scalar multiplication defined by $(\alp \tet)(a) = \tet(\alp a)$ for all $\alp\in \fq$, all $\tet \in \AC(q)$, and all $a \in \CA(q)$, and that $\{\seq{\tet\circ e^{\ast}}{r}\}$ is an $\fq$-basis of $\AC(q)$. We now prove the following elementary result. 

\begin{proposition} \label{vectorspace}
The Serre dual $\AC$ has the structure of a vector space over $\k$ with respect to the scalar multiplication defined by $$\alp [\tet] = [\alp (\tet \circ \tr_{mn,n})]$$ whenever $m,n \in \N$ are such that $\alp\in \F_{q^{m}}$ and $\tet \in \AC(q^{n})$. Furthermore, the set $\{[\vsig \circ e_{1}^{\ast}], \ldots, [\vsig \circ e_{r}^{\ast}]\}$ is a $\k$-basis of $\CA^{\ast}$.
\end{proposition}

\begin{proof}
Observing that $\tr_{m,m'} \circ \tr_{n,m} = \tr_{n,m'}$ for all $n,m,m' \in \N$ with $m' \mid m \mid n$, it is a matter of straightforward calculations to prove that the scalar multiplication is well-defined.

On the other hand, let $\seq{\alp}{r} \in \k$ be such that $$\alp_{1} [\vsig \circ e_{1}^{\ast}] + \cdots + \alp_{r} [\vsig \circ e_{r}^{\ast}] = \mathbf{0}$$ where $\mathbf{0} \in \AC$ denotes the equivalence class which contains the trivial character $0_{n} = 1_{\CA(q^{n})}$ of $\CA(q^{n})$ for some (hence, for all) $n \in \N$. Let $n \in \N$ be such that $\seq{\alp}{r} \in \F_{q^{n}}$. Then $$\alp_{1} (\vsig \circ e_{1}^{\ast} \circ \tr_{n,1}) + \cdots + \alp_{r} (\vsig \circ e_{r}^{\ast} \circ \tr_{n,1}) = 0_{n},$$ and so $$\alp_{1} (\vsig \circ \operatorname{tr}_{n,1} \circ f_{1}^{\ast}) + \cdots + \alp_{r} (\vsig \circ \operatorname{tr}_{n,1} \circ f_{r}^{\ast}) = 0_{n}$$ where $\{\seq{f^{\ast}}{r}\}$ is the $\F_{q^{n}}$-basis of $\CA^{\ast}(q^{n})$ dual to the $\F_{q^{n}}$-basis $\{\seq{e}{r}\}$ of $\CA(q^{n})$, and where $\map{\operatorname{tr}_{n,1}}{\F_{q^{n}}}{\fq}$ is the usual trace map defined by $$\operatorname{tr}_{n,1}(\alp) = \alp + \alp^{q} + \cdots + \alp^{q^{n-1}}$$ for all $\alp \in \F_{q^{n}}$. Since $\map{\operatorname{tr}_{n,1}}{\CA(q^{n})}{\CA(q)}$ is surjective, $\vsig' = \vsig \circ \operatorname{tr}_{n,1}$ is a non-trivial character of $\F_{q^{n}}^{+}$, and thus $\{\seq{\vsig' \circ f}{r}\}$ is an $\F_{q^{n}}$-basis of $\AC(q^{n})$. This implies that $\alp_{1} = \ldots = \alp_{r} = 0$, and so $\{[\vsig \circ e_{1}^{\ast}], \ldots, [\vsig \circ e_{r}^{\ast}]\}$ is a linearly independent subset of $\AC$.

Finally, let $n \in \N$, and let $\tet \in \AC(q^{n})$. Let $\vsig' = \vsig \circ \operatorname{tr}_{n,1}$ and $\{\seq{f^{\ast}}{r}\}$ be as above. Then there are $\seq{\alp}{r} \in \F_{q^{n}}$ such that $$\tet = \alp_{1} (\vsig' \circ f_{1}^{\ast}) + \cdots + \alp_{r} (\vsig' \circ f_{r}^{\ast}).$$ Since $\operatorname{tr}_{n,1} \circ f_{i}^{\ast} = e_{i}^{\ast} \circ \tr_{n,1})$ for all $1 \leq i \leq r$, we conclude that $$\tet = \alp_{1} (\vsig \circ e_{1}^{\ast} \circ \tr_{n,1}) + \cdots + \alp_{r} (\vsig \circ e_{r}^{\ast} \circ \tr_{n,1}),$$ and this clearly implies that $$[\tet] = \alp_{1} [\vsig \circ e_{1}^{\ast}] + \cdots + \alp_{r} [\vsig \circ e_{r}^{\ast}].$$ It follows that $\{[\vsig \circ e_{1}^{\ast}], \ldots, [\vsig \circ e_{r}^{\ast}]\}$ is a spanning set of $\AC$, and this completes the proof.
\end{proof}

Since there is no danger of ambiguity, we may assume that $\seq{e}{r} \in \CA$, so that $\{\seq{e}{r}\}$ is a $\k$-basis of $\CA$. The following result is an obvious consequence of the previous result.

\begin{lemma} \label{isoserredual1}
The mapping $e_{i} \mapsto [\vsig \circ e_{i}^{\ast}]$ extends linearly to a $\k$-linear isomorphism $\map{\VPsi}{\CA}{\AC}$.
\end{lemma}

Henceforth, we fix a unitary character $\map{\widetilde{\vsig}}{\k}{\cx}$ of the additive group $\k^{+}$ which extends the (non-trivial) character $\map{\vsig}{\fq}{\cx}$ of $\fq^{+}$. Such an extension always exists; see for example \cite[Corollary~4.41]{Folland1995a}.

\begin{lemma}
Let $\CA^{\circ}$ denote the Pontryagin dual of the additive group $\CA^{+}$. For every $\alp \in \k$ and every $\tet \in \CA^{\circ}$, we define the scalar product $\alp\tet \in \CA^{\circ}$ by $(\alp\tet)(\bet) = \tet(\alp\bet)$ for all $\bet \in \k$. Then $\CA^{\circ}$ becomes a vector space over $\k$ with respect to the scalar multiplication defined by the mapping $(\alp,\tet) \mapsto \alp\tet$. Furthermore, if $\{\seq{\widetilde{e}^{\,\ast}}{r}\}$ is the $\k$-basis of $\CA^{\ast}$ dual to the $\k$-basis $\{\seq{e}{r}\}$ of $\CA$, then the subset $\{\seq{\widetilde{\vsig} \circ \widetilde{e}^{\,\ast}}{r}\}$ of $\CA^{\circ}$ is linearly independent over $\k$.
\end{lemma}

\begin{proof}
The proof of the first assertion is straightforward. On the other hand, the mapping $f \mapsto \widetilde{\vsig} \circ f$ clearly defines an injective $\k$-linear map $\map{\widetilde{\vsig}^{\,\ast}}{\CA^{\ast}}{\CA^{\circ}}$. Since $\{\seq{\widetilde{e}^{\,\ast}}{r}\}$ is the $\k$-basis of $\CA^{\ast}$, it follows that $\{\widetilde{\vsig}^{\,\ast}(\widetilde{e}^{\,\ast}_{1}), \ldots, \widetilde{\vsig}^{\,\ast}(\widetilde{e}^{\,\ast}_{r})\}$ is a linearly independent subset of $\CA^{\circ}$, as required.
\end{proof}

\begin{proposition} \label{isoserredual2}
The mapping $e_{i} \mapsto \widetilde{\vsig} \circ \widetilde{e}^{\,\ast}_{i}$ extends linearly to an injective $\k$-linear map $\map{\VPhi}{\CA}{\CA^{\circ}}$ whose image $\VPhi(\CA)$ is dense with respect to the topology of pointwise convergence in $\CA^{\circ}$. Furthermore, the composition $\Lam = \VPhi \circ \VPsi\inv$ is an injective $\k$-linear map $\map{\Lam}{\AC}{\CA^{\circ}}$ whose image $\Lam(\AC)$ is dense in $\CA^{\circ}$.
\end{proposition}

\begin{proof}
By the previous lemma, it is clear that the linear map $\map{\VPhi}{\CA}{\CA^{\circ}}$ is injective. It is also clear that $\VPhi(\CA) = \set{\widetilde{\vsig} \circ f}{f \in \CA^{\ast}}$. On the other hand, let $\tet \in \CA^{\circ}$ be arbitrary, and for every $n \in \N$, let $\tet_{n} \in \AC(q^{n})$ denote the restriction $\tet|_{\AC(q^{n})}$ of $\tet$ to $\CA(q^{n})$. Then $\tet(a) = \lim_{n\to\infty} \tet_{n}(a)$ for all $a \in \CA$. Let $n \in \N$ be arbitrary, and observe that, since $\widetilde{\vsig}|_{\CA(q)} = \vsig$ is non-trivial, the restriction $\widetilde{\vsig}|_{\CA(q^{n})}$ is also non-trivial. Thus, every character of $\CA(q^{n})$ is of the form $\widetilde{\vsig}|_{\CA(q^{n})} \circ f = \widetilde{\vsig} \circ f$ for some $f \in \CA^{\ast}(q^{n})$. In particular, there exists $f_{n} \in \CA^{\ast}(q^{n})$ such that $\tet_{n} = \widetilde{\vsig} \circ f_{n}$. Now, if $\{e_{1,n}^{\ast}, \ldots, e_{r,n}^{\ast}\}$ is the $\F_{q^{n}}$-basis of $\CA^{\ast}(q^{n})$ dual to the $\F_{q^{n}}$-basis $\{\seq{e}{r}\}$ of $\CA(q^{n})$, then $f_{n} = \alp_{1} e_{1,n}^{\ast} + \cdots + \alp_{r} e_{r,n}^{\ast}$ for some $\seq{\alp}{r} \in \F_{q^{n}}$. Therefore, if we define $\widetilde{f}_{n} = \alp_{1} \widetilde{e}^{\,\ast}_{1} + \cdots + \alp_{r} \widetilde{e}^{\,\ast}_{r} \in \CA^{\ast}$, then $\widetilde{\vsig} \circ \widetilde{f}_{n} \in \VPhi(\CA)$ and $\tet_{n} = (\widetilde{\vsig} \circ \widetilde{f}_{n})|_{\CA(q^{n})}$. It follows that $$\tet(a) = \lim_{n\to\infty} \tet_{n}(a) = \lim_{n\to\infty} (\widetilde{\vsig} \circ \widetilde{f}_{n})(a)$$ for all $a \in \CA$, and this concludes the proof. (For the last assertion, it is enough to recall that $\map{\VPsi}{\CA}{\AC}$ is a $\k$-linear isomorphism.)
\end{proof}

We next consider the natural (continuous) action of $G$ on the left of $\CA^{\circ}$ given by $(g\tet)(a) = \tet(g\inv a)$ for all $g \in G$, all $\tet \in \CA^{\circ}$ and all $a\in \CA$. Using the injective $\k$-linear map $\map{\Lam}{\AC}{\CA^{\circ}}$, we may define an action of $G$ on the left of $\AC$ by the rule $g\btet = g\Lam(\btet)$ for all $g \in G$ and all $\btet \in \AC$. It is straightforward to check that $$g[\tet] = [g(\tet \circ \tr_{mn,n})]$$ whenever $g \in G(q^{m})$ for some $m \in \N$, and $\tet \in \AC(q^{n})$ for some $n \in \N$. We denote by $\bOme$ the set consisting of all $G$-orbits $G\btet$ for $\btet \in \AC$.

On the other hand, 
let $\CC^{0}(G)$ denote the complex vector space consisting of all continuous functions $G \to \C$, and consider the (continuous) action of $G$ on the left of $\CC^{0}(G)$ given by $(g\phi)(h) = \phi(g\inv h)$ for all $g,h \in G$ and all $\phi \in \CC^{0}(G)$. Hence, $\CC^{0}(G)$ becomes a left $\C[G]$-module which is unitary with respect to the usual Hermitian inner product. Similarly, let $\CC^{0}(\CA)$ denote the complex vector space of all continuous functions $\CA \to \C$, and observe that there is an obvious $\C$-linear isomorphism $\map{\eps}{\CC^{0}(\CA)}{\CC^{0}(G)}$ where, for every $\phi \in \CC^{0}(G)$, the function $\eps(\phi) \in \CC^{0}(\CA)$ is defined by $\eps(\phi)(g) = \phi(g-1)$ for all $g \in G$. In particular, since the Pontryagin dual $\CA^{\circ}$ is a linearly independent subset of $\CC^{0}(\CA)$, we see that $\eps(\CA^{\circ})$ is a linearly independent subset of $\CC^{0}(G)$, and this implies that $\eps(\Lam(\AC))$ is also a linearly independent subset of $\CC^{0}(G)$. For simplicity of writing, since there is no danger of confusion, we will henceforth identify $\AC$ with $\Lam(\AC)$, so that we realise $\AC$ as a (dense discrete) subset of $\CA^{\circ}$.

Now, let $\CV$ denote the vector subspace of $\CC^{0}(G)$ spanned by $\eps(\AC)$. It is not hard to check that $$g\eps(\tet) = \tet(g\inv-1)\eps(g\tet)$$ for all $g \in G$ and all $\tet \in \CA^{\circ}$, and thus $\eps(\CA^{\circ})$ is a $G$-invariant subset of $\CC^{0}(G)$. In particular, we see that the subset $\eps(\AC)$ of $\CC^{0}(G)$ is also $G$-invariant, and so $\CV$ is a $G$-invariant vector subspace (hence, a $\C[G]$-submodule) of $\CC^{0}(G)$. Furthermore, since $\eps(\AC)$ decomposes as the disjoint union of the distinct $G$-orbits $G\eps(\btet) = \eps(G\btet)$ for $\btet \in \AC$, the vector subspace $\CV$ decomposes as a direct sum $$\CV = \bigoplus_{\bome \in \bOme} \CV_{\bome}$$ where, for every $\bome \in \bOme$,
\begin{equation}\label{super}
\CV_{\bome} = \sum_{\btet \in \bome} \C\eps(\btet)
\end{equation}
is  the vector subspace of $\CC^{0}(G)$ linearly spanned by the $G$-orbit $\eps(\bome) \sset \CC^{0}(G)$. It is obvious that, for every $\bome \in \bOme$, the subspace $\CV_{\bome}$ is $G$-invariant, and hence $\CV_{\bome}$ is a $\C[G]$-submodule of $\CV$ which we call the \textit{super $\C[G]$-module} of $G$ associated with $\bome$.

Let $\btet \in \AC$ (viewed as an element of $\CA^{\circ}$), and let $L_{G}(\btet) = \set{g \in G}{g\btet = \btet}$ be the centraliser of $\btet$ in $G$. Then it is easy to prove that $L_{G}(\btet)$ is an algebra subgroup of $G$, and that the function $\eps(\btet) \in \CC^{0}(\CA)$ defines by restriction a linear (unitary) character $\map{\nu_{\btet}}{L_{G}(\btet)}{\cx}$ of $L_{G}(\btet)$. Let $\C_{\btet}$ denote a one-dimensional $\C[L_{G}(\btet)]$-module which affords $\nu_{\btet}$, and observe that $L_{G}(\btet)$ is the stabiliser of $\C_{\btet}$. On the other hand, let $\bome = G\btet$ be the $G$-orbit of $\btet$, and consider the $\C[G]$-module $\CV_{\bome}$. Then $G$ acts transitively on the basis $\eps(\bome)$ of $\CV_{\bome}$, and so there exists an isomorphism of $\C[G]$-modules $$\CV_{\bome} \cong \C[G] \ox_{\C[L_{G}(\btet)]} \C_{\btet} = \ind^{G}_{L_{G}(\btet)}(\C_{\btet}).$$

Now, suppose that $\btet = [\tet]$ for some $\tet \in \AC(q^{n})$ and some $n \in \N$. It is not hard to check that $\tet$ equals the restriction $\btet|_{\CA(q^{n})}$ of $\btet \in \CA^{\ast}$ to $\CA(q^{n})$. Furthermore, an argument similar to the one used in the proof of \refp{invariant} shows that the $G(q^{n})$-orbit $G(q^{n})\tet$ is precisely the intersection $G\btet \cap \AC(q^{n})$ which we will denote by $\bome(q^{n})$. As above, we may associate with this $G(q^{n})$-orbit a $\C[G(q^{n})]$-module $\CV_{\bome(q^{n})}$, which turns out to be isomorphic to the induced module $\ind^{G(q^{n})}_{L_{G(q^{n})}(\tet)}(\C_{\tet})$ where $\C_{\tet}$ denotes a one-dimensional $\C[L_{G(q^{n})}(\tet)]$-module which affords the linear character $\map{\nu_{\tet}}{L_{G(q^{n})}(\tet)}{\cx}$ of $L_{G(q^{n})}(\tet)$. In particular, it follows that $\CV_{\bome(q^{n})}$ affords the supercharacter $$\xi_{\tet} = \ind^{G(q^{n})}_{L_{G(q^{n})}(\tet)}(\nu_{\tet})$$ of $G(q^{n})$ (recall \refeq{supch} for the definition of $\xi_{\tet}$).

Therefore, we conclude that, for all $n \in \N$ such that $\btet = [\tet]$ for some $\tet \in \AC(q^{n})$, the $\C[G]$-module $\CV_{\bome}$ contains in its restriction to $G(q^{n})$ an isomorphic copy of the $\C[G(q^{n})]$-module $\CV_{\bome(q^{n})}$ which affords the supercharacter $\xi_{\tet}$ of $G(q^{n})$. In fact, we have the following result (where we identify $\AC$ with the subset $\Lam(\AC)$ of $\CA^{\circ}$).

\begin{theorem} \label{thm:super}
Let $\bxi \in \SCh(G)$ be arbitrary, and let $n \in \N$ be such that $\bxi = [\xi]$ for some supercharacter $\xi \in \SCh(G(q^{n}))$. Let $\btet = [\tet] \in \AC$ where $\tet \in \AC(q^{n})$ is such that $\xi = \xi_{\tet}$, and define $\CV(\bxi)$ to be the isomorphism class of the super $\C[G]$-module $\CV_{\bome}$ associated with the $G$-orbit $\bome = G\btet \sset \AC$. Then the mapping $\bxi \mapsto \CV(\bxi)$ induces a well-defined bijection $\map{\varPi}{\SCh(G)}{\SMod(G)}$ where we denote by $\SMod(G)$ the set consisting of all isomorphism classes of super $\C[G]$-modules.
\end{theorem}

\begin{proof}
On one hand, let $n' \in \N$ and let $\tet' \in \AC(q^{n'})$ be such that $\bxi = [\xi']$ where $\xi' = \xi_{\tet'} \in \SCh(G(q^{n'}))$. Without loss of generality, we may assume that $n \mid n'$. By \reft{main}, we see that $\xi_{\tet} = \xi_{\tet' \circ \tr_{n',n}}$, and so $\tet'\circ\tr_{n',n} = g\tet h$ for some $g,h \in G(q^{n})$. It follows that $\btet' = [\tet'] = [g\tet h] = g\btet h$ where we define the $(G\x G)$-action on $\AC$ is defined in a way completely analogous to the $G$-action. It follows that $\btet' \in G\btet h = \bome h$, and thus the $\C[G]$-module $\CV_{\bome'} = \CV_{\bome h}$ is isomorphic to $\CV_{\bome}$. Indeed, it is straightforward to check that $\eps(g\tet h) = \ovl{(g\tet)(h\inv-1)} (\eps(g\tet) h)$ for all $g \in G$. Therefore, $\CV(\bxi)$ does not depend on the choice of the representative of $\bxi$, and so the map $\map{\varPi}{\SCh(G)}{\SMod(G)}$ is well-defined. It is clearly surjective, and thus it remains to prove that it is injective.

Let $\bome, \bome' \in \bOme$ be such that the super $\C[G]$-modules $\CV_{\bome}$ and $\CV_{\bome'}$ are isomorphic, and let $\btet, \btet' \in \AC$ be such that $\bome = G\btet$ and $\bome' = G\btet'$. Moreover, let $n,n' \in \N$ be such that $\btet = [\tet]$ and $\btet' = [\tet']$ for some $\tet \in \AC(q^{n})$ and some $\tet' \in \AC(q^{n'})$. Then we must prove that the supercharacters $\xi_{\tet} \in \SCh(G(q^{n}))$ and $\xi_{\tet'} \in \SCh(G(q^{n'}))$ lie in the same equivalence class $\bxi \in \SCh(G)$.

We first prove that there exists a $\C$-linear isomorphism $$\Hom_{\C[G]}(\CV_{\bome},\CV_{\bome'}) \cong \CW_{\bome} \cap \CV_{\bome'}$$ where $\CW_{\bome}$ denotes the vector subspace of $\CV$ linearly spanned by $\eps(\btet G)$. To see this, let $(\CV_{\bome})^{\perp} = \bigoplus_{\bome'' \in \bOme \setminus \{\bome\}} \CV_{\bome''}$. Then $(\CV_{\bome})^{\perp}$ is a $\C[G]$-submodule of $\CV$ satisfying $\CV = \CV_{\bome} \+ (\CV_{\bome})^{\perp}$, and every $\vphi \in \Hom_{\C[G]}(\CV_{\bome},\C[G])$ can be naturally extended to an endomorphism $\widetilde{\vphi} \in \End_{\C[G]}(\C[G])$ satisfying $(\CV_{\bome})^{\perp} \sset \ker(\widetilde{\vphi})$, and thus there exists a unique $z \in \C[G]$ such that $\widetilde{\vphi}(a) = az$ for all $a \in \C[G]$. In particular, we see that $\vphi(\eps(\btet)) = \eps(\btet)z \in \CW_{\bome}$, and so the mapping $\vphi \mapsto \vphi(\eps(\btet))$ defines a $\C$-linear isomorphism $\Hom_{\C[G]}(\CV_{\bome},\C[G]) \cong \CW_{\bome}$. Since $\vphi \in \Hom_{\C[G]}(\CV_{\bome},\CV_{\bome'})$ if and only if $\vphi(\eps(\btet)) \in \CW_{\bome} \cap \CV_{\bome'}$, we conclude that the mapping $\vphi \mapsto \vphi(\eps(\btet))$ restricts to a $\C$-linear isomorphism $\Hom_{\C[G]}(\CV_{\bome},\CV_{\bome'}) \cong \CW_{\bome} \cap \CV_{\bome'}$.

Now, since the super $\C[G]$-modules $\CV_{\bome}$ and $\CV_{\bome'}$ are isomorphic, $\CW_{\bome} \cap \CV_{\bome'} \neq \{0\}$. Since $\CW_{\bome} \cap \CV_{\bome'}$ is linearly spanned by the intersection $G\btet \cap \btet G$, we conclude that $\btet' = g\btet h$ for some $g,h \in G$. Finally, we may clearly choose $n'' \in \N$ such that $g,h \in G(q^{n''})$ and $\btet = [\tet'']$ for some $\tet'' \in \AC(q^{n''})$. Then $\btet' = g\btet h = [g\tet'' h]$, and so $\tet' \sim g\tet'' h$ (recall that $\btet' = [\tet']$). By \reft{main}, we deduce that $\xi_{\tet'} \sim \xi_{g\tet'' h} = \xi_{\tet''} \sim \xi_{\tet}$ (because $\btet = [\tet''] = [\tet]$), and thus $[\xi_{\tet'}] = [\xi_{\tet}]$, as required.
\end{proof}



\bibliographystyle{acm}

\end{document}